\documentclass[reqno,A4paper]{amsart}
\usepackage{amsmath,amsthm,amssymb,amsfonts,amscd}
\usepackage{enumerate}
\usepackage[noadjust]{cite}
\usepackage{diagmac2} 
\usepackage{tikz}
\usetikzlibrary{positioning,calc,arrows,arrows.meta}

\usepackage{eqlist}
\usepackage{array}

\usepackage{color}

\newtheorem{theorem}{Theorem}[section]
\newtheorem{corollary}[theorem]{Corollary}
\newtheorem{lemma}[theorem]{Lemma}
\newtheorem{proposition}[theorem]{Proposition}

\theoremstyle{definition}
\newtheorem{definition}[theorem]{Definition}
\newtheorem{example}[theorem]{Example}
\newtheorem{problem}[theorem]{Problem}

\theoremstyle{remark}
\newtheorem{remark}[theorem]{Remark}

\numberwithin{equation}{section}

\makeatletter

\renewcommand{\p@enumii}{}
\makeatother

\DeclareMathOperator{\diam}{diam}

\newcommand{\RR}{\mathbb{R}}

\newcommand{\Id}{\operatorname{Id}}

\def\<#1>{\langle #1 \rangle}

\newbox\onebox
\newcommand{\coherent}[1]{\mathbin{\setbox\onebox=\hbox{$=$}\lower0.7\ht%
\onebox\hbox{$\stackrel{#1}{=}$}}}

\newcommand{\acr}{\newline\indent}

\begin{document}

\title{Ultrametric preserving functions and weak similarities of ultrametric spaces}
\author{Viktoriia Bilet}
\address{Institute of Applied Mathematics and Mechanics of NASU\acr
Dobrovolskogo str. 1, Slovyansk 84100, Ukraine}
\email{viktoriiabilet@gmail.com}
\thanks{Viktoriia Bilet and Oleksiy Dovgoshey were partially supported in the frame of project 0117U002165: Development of Mathematical Models, Numerical and Analytical Methods, and Algorithms for Solving Modern Problems of Biomedical Research.}

\author{Oleksiy Dovgoshey}
\address{Institute of Applied Mathematics and Mechanics of NASU\acr
Dobrovolskogo str. 1, Slovyansk 84100, Ukraine}
\email{oleksiy.dovgoshey@gmail.com}

\author{Ruslan Shanin}
\address{Odesa I. I. Mechnikov National University\acr
Dvoryanskaya str. 2, Odesa, 65082, Ukraine}
\email{ruslanshanin@gmail.com}


\begin{abstract}
Let \(WS(X, d)\) be the class of ultrametric spaces which are weakly similar to ultrametric space \((X, d)\). The main results of the paper completely describe the ultrametric spaces \((X, d)\) for which the equality
\[
\rho(x, y) = f(d(\Phi(x), \Phi(y)))
\]
holds for every \((Y, \rho) \in WS(X, d)\), every weak similarity \(\Phi \colon Y \to X\), and all \(x\), \(y \in Y\) with some ultrametric (pseudoultrametric) preserving function \(f\) depending on \(\Phi\).
\end{abstract}

\maketitle

\section{Introduction and basic definitions}

Recall some definitions from the theory of metric spaces. In what follows we write \(\RR^{+}\) for the set of all nonnegative real numbers.

\begin{definition}\label{d1.1}
A \textit{metric} on a set \(X\) is a function \(d\colon X\times X\rightarrow \RR^{+}\) such that for all \(x\), \(y\), \(z \in X\):
\begin{enumerate}
\item \((d(x,y) = 0) \Leftrightarrow (x=y)\);
\item \(d(x,y)=d(y,x)\);
\item \(d(x, y)\leq d(x, z) + d(z, y)\).
\end{enumerate}
A metric \(d\colon X\times X\rightarrow \RR^{+}\) is an ultrametric on \(X\) if
\begin{enumerate}
\item [\((iv)\)] \(d(x,y) \leq \max \{d(x,z),d(z,y)\}\)
\end{enumerate}
holds for all \(x\), \(y\), \(z \in X\).
\end{definition}

Inequality \((iv)\) is often called the {\it strong triangle inequality}. 

Let \(A\) be a subset of a metric space \((X, d)\), \(A \subseteq X\). The quantity
\begin{equation}\label{e2.1}
\diam A = \diam (A, d) =\sup\{d(x,y) \colon x, y\in A\}
\end{equation}
is the \emph{diameter} of \(A\). If the inequality \(\diam A < \infty\) holds, then we say that \(A\) is a \emph{bounded} subset of \(X\).

We define the \emph{distance set} \(D(X)\) of a metric space \((X,d)\) as
\[
D(X) = D(X, d) = \{d(x, y) \colon x, y \in X\}.
\]
It is clear that \(\diam X = \sup D(X)\).

The theory of ultrametric spaces is closely connected with various of investigations in mathematics, physics, linguistics, psychology and computer science. Some properties of ultrametric spaces have been studied in~\cite{DM2009, DD2010, DP2013SM, Groot1956, Lemin1984FAA, Lemin1984RMS39:5, Lemin1984RMS39:1, Lemin1985SMD32:3, Lemin1988, Lemin2003, Qiu2009pNUAA, Qiu2014pNUAA, BS2017, DM2008, DLPS2008TaiA, KS2012, Vau1999TP, Ves1994UMJ, Ibragimov2012, GH1961S, PTAbAppAn2014}. Note that the use of trees and tree-like structures gives a natural language for description of ultrametric spaces \cite{Carlsson2010, DLW, Fie, GV2012DAM, HolAMM2001, H04, BH2, Lemin2003, Bestvina2002, DDP2011pNUAA, DP2019PNUAA, DPT2017FPTA, PD2014JMS, DPT2015, Pet2018pNUAA,DP2018pNUAA}.

A function \(f \colon \RR^{+} \to \RR^{+}\) is called \emph{ultrametric preserving} if
\[
X \times X \xrightarrow{d} \RR^{+} \xrightarrow{f} \RR^{+}
\]
is an ultrametric on \(X\) for every ultrametric space \((X, d)\).

The following elegant theorem was obtained by Pongsriiam and Termwuttipong in~\cite{PTAbAppAn2014} and this can be considered as an initial motivation for present paper. 

\begin{theorem}\label{t5.21}
The following conditions are equivalent for every \(f \colon \RR^{+} \to \RR^{+}\):
\begin{enumerate}
\item \label{t5.21:s1} \(f\) is ultrametric preserving.
\item \label{t5.21:s2} \(f\) is increasing and the equality \(f(x) = 0\) holds if and only if \(x = 0\).
\end{enumerate}
\end{theorem}

Some facts related to ultrametric preserving functions can be found in \cite{Dov2020MS, SKP2020MS, VD2019a, PTAbAppAn2014}.

\begin{example}\label{ex5.24}
Let \(f \colon \RR^{+} \to \RR^{+}\) be defined as
\begin{equation}\label{ex5.24:e1}
f(t) = \frac{d^{*} t}{1+t}
\end{equation}
with \(d^{*} \in (0, \infty)\). Then \(f\) is increasing and \(f(t) = 0\) holds if and only if \(t=0\). Hence, \(f\) is ultrametric preserving by Theorem~\ref{t5.21}. Moreover, \(f \circ d\) is a metric for every metric space \((X, d)\) (Example~2 in \cite{Dobos1998}). Thus, \(f\) is also \emph{metric preserving}. Using the function \(f\) it is easy to show that, for every unbounded metric space \((Y, \rho)\), the metric space \((Y, \delta)\) with \(\delta = f \circ \rho\) is bounded and has the same topology as \((Y, \rho)\). 
\end{example}

The metric-preserving functions were detailed studied by J.~Dobo\v{s} and other mathematicians \cite{Wilson1935, BDMS1981, Borsik1988, Corazza1999, Das1989, Dobos1994, Dobos1995, Dobos1996, Dobos1996a, Dobos1998, Piotrowski2003, Pokorny1993, Termwuttipong2005, Vallin2000, Khemar2007, BFS2003BazAaG, Foertsch2002, HMCM1991, DM2013, DPK2014MS, KP2018MS, KPS2019IJoMaCS} but the properties of functions which preserve special type metrics remain little studied.

The following definition was introduced in \cite{DP2013AMH}.

\begin{definition}\label{d2.34}
Let \((X, d)\) and \((Y, \rho)\) be nonempty metric spaces. A mapping \(\Phi \colon X \to Y\) is a \emph{weak similarity} of \((X, d)\) and \((Y, \rho)\) if \(\Phi\) is bijective and there is a strictly increasing function \(\psi \colon D(Y) \to D(X)\) such that the following diagram
\begin{equation}\label{d2.34:e2}
\ctdiagram{
\ctv 0,25:{X \times X}
\ctv 100,25:{Y \times Y}
\ctv 0,-25:{D(X)}
\ctv 100,-25:{D(Y)}
\ctet 0,25,100,25:{\Phi \otimes \Phi}
\ctel 0,25,0,-25:{d}
\cter 100,25,100,-25:{\rho}
\ctet 100,-25,0,-25:{\psi}
}
\end{equation}
is commutative, i.e., the equality
\begin{equation}\label{d2.34:e1}
d(x, y) = \psi\left(\rho\bigl(\Phi(x), \Phi(y)\bigr)\right)
\end{equation}
holds for all \(x\), \(y \in X\).
\end{definition}

\begin{remark}\label{r5.32}
In Diagram~\eqref{d2.34:e2} (and other diagrams illustrating applications of weak similarities) we write \(\Phi \otimes \Phi\) for the mapping
\[
X \times X \ni \<x, y> \mapsto \<\Phi(x), \Phi(y)> \in Y \times Y.
\]
Moreover, for every metric \(d\) on \(X\) we use the symbol \(d \colon X \times X \to D(X)\) for the surjection induced by restricting codomain \(\RR^{+}\) of \(d\) to its range \(D(X)\).
\end{remark}

If \(\Phi \colon X \to Y\) is a weak similarity with commutative diagram~\eqref{d2.34:e2}, then we say that \((X, d)\) and \((Y, \rho)\) are \emph{weakly similar}, and \(\psi\) is the \emph{scaling function} of \(\Phi\). For every ultrametric space \((Z, \delta)\) we will denote by \(WS(Z, \delta)\) the class of all ultrametric spaces which are weakly similar to \((Z, \delta)\).

\begin{remark}\label{r5.29}
It follows directly from Definition~\ref{d2.34} that every scaling function is bijective and uniquely determined by corresponding weak similarity.
\end{remark}

\begin{example}\label{ex5.29}
Let \((X, d)\) and \((Y, \rho)\) be nonempty metric spaces. A mapping \(\Phi \colon X \to Y\) is a \emph{similarity}, if \(\Phi\) is bijective and there is a strictly positive number \(r\), the \emph{ratio} of \(\Phi\), such that
\[
\rho\bigl(\Phi(x), \Phi(y)\bigr) = rd(x, y)
\]
for all \(x\), \(y \in X\) (see, for example, \cite[p.~45]{Edgar1992}). It is clear that every isometry is a similarity with the ratio \(r = 1\) and every similarity is a weak similarity. Furthermore, a weak similarity \(\Phi \colon X \to Y\) with the scaling function \(\psi \colon D(Y) \to D(X)\) is an isometry if and only if \(\psi(t) = t\) holds for every \(t \in D(Y)\).
\end{example}

\begin{example}\label{ex5.32}
Let \(\Phi \colon X \to Y\) be a bijection and let \(d \colon X \times X \to \RR^{+}\) and \(\rho \colon Y \times Y \to \RR^{+}\) be some metrics. The mapping \(\Phi\) is a weak similarity of \((X, d)\) and \((Y, \rho)\) if and only if the equivalence
\begin{equation}\label{ex5.32:e1}
(d(x, y) \leqslant d(w, z)) \Leftrightarrow (\rho(\Phi(x), \Phi(y)) \leqslant \rho(\Phi(w), \Phi(z)))
\end{equation}
is valid for all \(x\), \(y\), \(z\), \(w \in X\).
\end{example}

\begin{remark}\label{r5.36}
Equivalence~\eqref{ex5.32:e1} evidently implies the validity of 
\begin{equation}\label{r5.36:e1}
(d(x, y) = d(w, z)) \Leftrightarrow (\rho(\Phi(x), \Phi(y)) = \rho(\Phi(w), \Phi(z))).
\end{equation}
The bijections \(\Phi \colon X \to Y\) satisfying~\eqref{r5.36:e1} for all \(x\), \(y\), \(z\), \(w \in X\) is said to be \emph{combinatorial similarities}. Some questions connected with the weak similarities and combinatorial similarities were studied in \cite{DovBBMSSS2020, DLAMH2020, Dov2019IEJA}. The weak similarities of finite ultrametric and semimetric spaces were also considered by E.~Petrov in \cite{Pet2018pNUAA}.
\end{remark}

The useful generalization of the concept of ultrametric is the concept of pseudoultrametric.

\begin{definition}\label{d5.43}
Let \(X\) be a set and let \(d \colon X \times X \to \RR^{+}\) be a symmetric function such that \(d(x, x) = 0\) holds for every \(x \in X\). The function \(d\) is a \emph{pseudoultrametric} on \(X\) if it satisfies the strong triangle inequality.
\end{definition}

If \(d\) is a pseudoultrametric on \(X\), then we will say that \((X, d)\) is a \emph{pseudoultrametric} \emph{space}.

Every ultrametric space is a pseudoultrametric space but not conversely. In contrast to ultrametric spaces, pseudoultrametric spaces can contain some distinct points with zero distance between them.

\begin{definition}\label{d5.44}
A function \(f \colon \RR^{+} \to \RR^{+}\) is \emph{pseudoultrametric preserving} if \(f \circ d\) is a pseudoultrametric for every pseudoultrametric space \((X, d)\).
\end{definition}

In the case of pseudoultrametric preserving functions, the Pongsriiam---Termwuttipong theorem can be modified as follows.

\begin{proposition}\label{p5.45}
The following conditions are equivalent for every \(f \colon \RR^{+} \to \RR^{+}\):
\begin{enumerate}
\item \label{p5.45:s1} The function \(f\) is increasing and \(f(0) = 0\) holds.
\item \label{p5.45:s2} The function \(f\) is pseudoultrametric preserving.
\end{enumerate}
\end{proposition}

For the proof see Proposition~2.4 \cite{Dov2020MS}.

\begin{remark}\label{r5.47}
Theorem~\ref{t5.21} and Proposition~\ref{p5.45} imply that every ultrametric preserving function is pseudoultrametric preserving. Moreover, a pseudoultrametric preserving \(f \colon \RR^{+} \to \RR^{+}\) is ultrametric preserving if and only if \(f^{-1}(0) = \{0\}\) holds.
\end{remark}

The goal of the present paper is a description of interrelations between ultrametric preserving functions and weak similarities of ultrametric spaces. In particular, we solve the following.

\begin{problem}\label{pr1.13}
Describe the ultrametric spaces \((X, d)\) for which every \((Z, \delta) \in WS(X, d)\) and every weak similarity \(\Phi \colon Z \to X\) admit an ultrametric preserving (pseudoultrametric preserving or strictly increasing ultrametric preserving) function \(f \colon \RR^{+} \to \RR^{+}\) such that the diagram
\begin{equation}\label{pr1.13:e1}
\ctdiagram{
\ctv 0,25:{X \times X}
\ctv 100,25:{D(X)}
\ctv 200,25:{\RR^{+}}
\ctv 0,-25:{Z \times Z}
\ctv 100,-25:{D(Z)}
\ctv 200,-25:{\RR^{+}}
\ctet 0,25,100,25:{d}
\ctet 100,25,200,25:{\Id_{D(X)}}
\ctet 0,-25,100,-25:{\delta}
\ctet 100,-25,200,-25:{\Id_{D(Z)}}
\ctel 0,-25,0,25:{\Phi \otimes \Phi}
\ctel 100,25,100,-25:{\psi}
\cter 200,25,200,-25:{f}
}
\end{equation}
is commutative, where \(\psi\) is the scaling function of \(\Phi\), \(\Id_{D(X)}\), \(\Id_{D(Z)}\) are the identical embeddings of the distance sets \(D(X)\) and, respectively, \(D(Z)\) in \(\RR^{+}\).
\end{problem}

The main results of the paper, Theorem~\ref{t5.46}, Theorem~\ref{t5.51} and Theorem~\ref{t5.34}, give a solution of Problem~\ref{pr1.13} for the case when \(f\) is pseudoultrametric preserving, ultrametric preserving and strictly increasing ultrametric preserving, respectively. Moreover, in Theorems~\ref{t2.25} and \ref{t2.35} we characterize bounded ultrametric spaces which are weakly similar to unbounded ones.

\begin{remark}\label{r5.67}
The concept of weak similarity is closely related to the ordinal scaling, multidimensional scaling and ranking that appear naturally in many applied researches \cite{Agarwal2007, Borg2005, Jamieson2011, Kleindessner2014, KruP1964, QYJ2004, Rosales2006, SheP1962, SheJoMP1966, Wauthier2013}. The presence of these relationships and, thus, the existence of potential applications, makes the study of weak similarities more important and promising.
\end{remark}

\section{Extension of scaling functions to ultrametric preserving functions}

The above formulated Pongsriiam---Termwuttipong theorem can be partially generalized as follows.

\begin{lemma}\label{l5.22}
Let \((X, d)\) be a nonempty ultrametric space with the distance set \(D(X)\) and let \(f \colon D(X) \to \RR^{+}\) be an increasing function. Then the mapping
\[
X \times X \ni \<x, y> \mapsto f(d(x, y)) \in \RR^{+}
\]
is an ultrametric on \(X\) if and only if the equation \(f(t) = 0\) has the unique solution \(t=0\).
\end{lemma}

\begin{proof}
It directly follows form Theorem~5.6 \cite{VD2019a}.
\end{proof}

\begin{theorem}\label{t2.25}
Let \((X, d)\) be an unbounded ultrametric space, let \(d^{*} \in (0, \infty)\) and \(\rho \colon X \times X \to \RR^{+}\) be defined as
\begin{equation}\label{t2.25:e1}
\rho(x, y) = \frac{d^{*} \cdot d(x, y)}{1 + d(x, y)}.
\end{equation}
Then \((X, \rho)\) is a bounded ultrametric space and the distance set \((D(X, \rho), {\leqslant})\) does not have the largest element.

Conversely, let \((X, \rho)\) be a nonempty ultrametric space such that \((D(X, \rho), {\leqslant})\) does not have the largest element. Write \(d^{*} = \diam (X, \rho)\). Then there is an unbounded ultrametric space \((X, d)\) such that~\eqref{t2.25:e1} holds for all \(x\), \(y \in X\).
\end{theorem}

\begin{proof}
It was noted in Example~\ref{ex5.24} that the function \(f\) defined by~\eqref{ex5.24:e1} is ultrametric preserving. Hence, the mapping \(\rho \colon X \times X \to \RR^{+}\), defined by~\eqref{t2.25:e1}, is an ultrametric. Moreover, since \(f\) is strictly increasing and satisfies the equality
\begin{equation*}
\lim_{t \to \infty} f(t) = d^{*},
\end{equation*}
we have
\[
\rho(x, y) < \lim_{t \to \infty} f(t) = d^{*} = \diam (X, \rho)
\]
for all \(x\), \(y \in X\). Thus, we have \(\diam (X, \rho) \notin D(X, \rho)\), i.e., \((D(X, \rho), {\leqslant})\) does not have the largest element.

Conversely, let \((X, \rho)\) be a bonded ultrametric space such that \(\diam (X, \rho) \notin D(X, \rho)\). Write \(d^{*} = \diam (X, \rho)\). The condition \(\diam (X, \rho) \notin D(X, \rho)\) and boundedness of \((X, \rho)\) imply that \(d^{*} \in (0, \infty)\). The function \(g \colon [0, d^{*}) \to \RR^{+}\),
\begin{equation}\label{t2.25:e7}
g(s) = \frac{s}{d^{*} - s},
\end{equation}
is strictly increasing and satisfies the equalities
\begin{equation}\label{t2.25:e5}
g(0) = 0 \quad \text{and} \quad \lim_{\substack{s \to d^{*} \\ s \in [0, d^{*})}} g(s) = +\infty.
\end{equation}
Consequently, there are sequences \((x_n)_{n \in \mathbb{N}}\) and \((y_n)_{n \in \mathbb{N}}\) of points of \(X\) such that
\begin{equation}\label{t2.25:e8}
\lim_{n \to \infty} \rho(x_n, y_n) = d^{*}.
\end{equation}
Since we have \(\rho(x, y) < d^{*}\) for all \(x\), \(y \in X\), the inclusion \(D(X, \rho) \subseteq [0, d^{*})\) holds. Now Lemma~\ref{l5.22} implies that the mapping \(d \colon X \times X \to \RR^{+}\) satisfying the equality
\[
d(x, y) = g(\rho(x, y))
\]
for all \(x\), \(y \in X\) is an ultrametric on~\(X\). From the second equality in \eqref{t2.25:e5} and equality~\eqref{t2.25:e8} it follows that \((X, d)\) is unbounded. A direct calculation shows
\begin{equation}\label{t2.25:e6}
f(g(s)) = s \quad \text{and} \quad g(f(t)) = t
\end{equation}
for all \(s \in [0, d^{*})\) and \(t \in [0, +\infty)\), where \(f\) is defined by~\eqref{ex5.24:e1}. Now equality \eqref{t2.25:e1} follows from \eqref{t2.25:e6}.
\end{proof}

\begin{lemma}\label{l5.29}
Let \((X, d)\), \((Y, \rho)\) and \((Z, \delta)\) be nonempty metric spaces and mappings \(\Phi \colon X \to Y\) and \(\Psi \colon Y \to Z\) be weak similarities with the scaling functions \(f \colon D(Y) \to D(X)\) and \(g \colon D(Z) \to D(Y)\), respectively. Then the mapping
\[
X \xrightarrow{\Phi} Y \xrightarrow{\Psi} Z
\]
is a weak similarity of \((X, d)\) and \((Z, \delta)\), and the function 
\[
D(Z) \xrightarrow{g} D(Y) \xrightarrow{f} D(X)
\]
is the scaling function of this weak similarity. Moreover, the inverse mapping \(\Phi^{-1} \colon Y \to X\) of \(\Phi\) is also a weak similarity, the scaling function \(f \colon D(Y) \to D(X)\) of \(\Phi\) is an order isomorphism of the posets \((D(Y), {\leqslant})\) and \((D(X), {\leqslant})\) such that the inverse isomorphism \(f^{-1} \colon D(X) \to D(Y)\) is the scaling function of the weak similarity \(\Phi^{-1}\).
\end{lemma}

\begin{proof}
The first part of the lemma follows from the commutativity of the diagram
\[
\ctdiagram{
\ctv 0,50:{X \times X}
\ctv 100,50:{Y \times Y}
\ctv 200,50:{Z \times Z}
\ctv 0,0:{D(X)}
\ctv 100,0:{D(Y)}
\ctv 200,0:{D(Z)}
\ctet 0,50,100,50:{\Phi \otimes \Phi}
\ctet 100,50,200,50:{\Psi \otimes \Psi}
\ctet 100,0,0,0:{f}
\ctet 200,0,100,0:{g}
\ctel 0,50,0,0:{d}
\ctel 100,50,100,0:{\rho}
\ctel 200,50,200,0:{\delta}
}.
\]

To prove the second part of the lemma, it suffices to use Remark~\ref{r5.29} and to note that that the equality
\[
d = f \circ \rho \circ (\Phi \otimes \Phi)
\]
holds if and only if we have \(f^{-1} \circ d \circ (\Phi^{-1} \otimes \Phi^{-1}) = \rho\).
\end{proof}

\begin{remark}\label{r5.31}
For more detailed proof of Lemma~\ref{l5.29} in more general case of semimetric spaces see Proposition~1.2 in \cite{DP2013AMH}.
\end{remark}

\begin{proposition}\label{p5.32}
Let \(f \colon \RR^{+} \to \RR^{+}\) be an ultrametric preserving function. Then the following conditions are equivalent:
\begin{enumerate}
\item \label{p5.32:s1} The function \(f\) is strictly increasing.
\item \label{p5.32:s2} The membership \((X, f \circ d) \in WS(X, d)\) is valid for every nonempty ultrametric space \((X, d)\).
\end{enumerate}
\end{proposition}

\begin{proof}
\(\ref{p5.32:s1} \Rightarrow \ref{p5.32:s2}\). Suppose \ref{p5.32:s1} holds. Let \((X, d)\) be a nonempty ultrametric space with the distance set \(D(X, d)\). Then \((X, f \circ d)\) is an ultrametric space with the distance set 
\[
D(X, f \circ d) = f(D(X, d)).
\]
Since \(f\) is strictly increasing, the identical function \(\operatorname{Id} \colon X \to X\) is a weak similarity of \((X, f \circ d)\) and \((X, d)\), and the restriction \(f|_{D(X, d)}\) is the corresponding scaling function,
\[
\ctdiagram{
\ctv 0,50:{X \times X}
\ctv 100,50:{X \times X}
\ctv 0,0:{D(X, f \circ d)}
\ctv 100,0:{D(X, d)}
\ctet 0,50,100,50:{\operatorname{Id} \otimes \operatorname{Id}}
\ctel 0,50,0,0:{f \circ d}
\cter 100,50,100,0:{d}
\ctet 100,0,0,0:{f|_{D(X, d)}}
}.
\]

\(\ref{p5.32:s2} \Rightarrow \ref{p5.32:s1}\). Suppose we have \ref{p5.32:s2} but \ref{p5.32:s1} does not hold. Then there are points \(a\), \(b \in \RR^{+}\) such that \(0 < a < b < \infty\) and \(f(a) = f(b)\). Let \(X = \{x_1, x_2, x_3\}\) be a tree-point set and let an ultrametric \(d \colon X \times X \to \RR^{+}\) satisfy the equalities
\[
d(x_1, x_2) = a, \quad d(x_1, x_3) = d(x_2, x_3) = b.
\]
Then \(f \circ d\) is an ultrametric on \(X\) such that 
\[
f(d(x_1, x_2)) = f(d(x_1, x_3)) = f(d(x_2, x_3)).
\]
Thus, we have
\begin{equation}\label{p5.32:e1}
|D(X, d)| = 3 \neq 2 = |D(X, f \circ d)|.
\end{equation}
It was noted in Remark~\ref{r5.29} that every scaling function is bijective. Hence, if \((X, d)\) and \((X, f \circ d)\) are weakly similar, then the equality
\[
|D(X, d)| = |D(X, f \circ d)|
\]
holds, contrary to \eqref{p5.32:e1}.
\end{proof}

Using the concept of weak similarity we can give a more compact variant of Theorem~\ref{t2.25}.

\begin{theorem}\label{t2.35}
Let \((X, d)\) be a nonempty ultrametric space. Then the following statements are equivalent:
\begin{enumerate}
\item\label{t2.35:s1} \((X, d)\) is weakly similar to an unbounded ultrametric space.
\item\label{t2.35:s2} The poset \((D(X, d), {\leqslant})\) does not have the largest element.
\end{enumerate}
\end{theorem}

\begin{proof}
\(\ref{t2.35:s1} \Rightarrow \ref{t2.35:s2}\). Let \((X, d)\) be a weakly similar to an unbounded ultrametric space \((Y, \rho)\). Then \(\diam (Y, \rho) \notin D(Y, \rho)\) holds. Hence, \(\diam (X, d) \notin D(X, d)\) by Lemma~\ref{l5.29}.

\(\ref{t2.35:s2} \Rightarrow \ref{t2.35:s1}\). Suppose \ref{t2.35:s2} holds. If \((X, d)\) is unbounded, then \(\ref{t2.35:s1}\) is valid because \((X, d)\) is weakly similar to itself. If \((X, d)\) is bounded, then, by Theorem~\ref{t2.25}, there is an unbounded ultrametric space \((X, \rho)\) such that
\[
d(x, y) = \diam (X, \rho) \cdot \frac{\rho(x, y)}{1 + \rho(x, y)}
\]
for all \(x\), \(y \in X\). The function \(f \colon \RR^{+} \to \RR^{+}\),
\[
f(t) = \diam (X, \rho) \cdot \frac{t}{1+t}
\]
is strictly increasing and ultrametric preserving. Hence, \((X, d)\) and \((X, \rho)\) are weakly similar by Proposition~\ref{p5.32}.
\end{proof}

Recall that a topological space \((X, \tau)\) is \emph{connected} if it is not the union of two disjoint nonempty open subsets of \(X\). A subset \(S\) of \(X\) is said to be connected if \(S\) is connected as subspace of \((X, \tau)\). A \emph{component} (= connected component) of \((X, \tau)\) is, by definition, a connected set \(A \subseteq X\) for which the implication
\[
(S \supseteq A) \Rightarrow (S = A)
\]
is valid for every connected \(S \subseteq X\). The set of all components of nonempty \((X, \tau)\) forms a partition of \(X\). Thus, for every \(x \in X\), there exists a unique component \(C(x)\) of \((X, \tau)\) such that \(x \in C(x)\). The set \(C(x)\) is called to be the \emph{component} of \(x\).

\begin{lemma}\label{l5.34}
A nonempty subset \(S\) of the real line \(\RR\) is connected if and only if \(S\) is an interval.
\end{lemma}

For the proof see Proposition~4 in \cite[p.~336]{Bourbaki1995}.

\begin{remark}\label{r5.35}
By bounded interval in \(\RR\) we mean any subset of \(\RR\) which can be given in the one of the following forms:
\begin{align*}
[a, b] &= \{x \in \RR \colon a \leqslant x \leqslant b\}, &  (a, b] &= \{x \in \RR \colon a < x \leqslant b\}, \\
[a, b) &= \{x \in \RR \colon a \leqslant x < b\}, & (a, b) &= \{x \in \RR \colon a < x < b\},
\end{align*}
where \(a\), \(b \in \RR\) with \(a \leqslant b\) for \([a, b]\) and \(a < b\) for \((a, b]\), \([a, b)\) and \((a, b)\). Similarly, by definition, any unbounded interval in \(\RR\) has one of the forms:
\begin{align*}
(-\infty, \infty) &= \RR, &  [a, \infty) &= \{x \in \RR \colon x \geqslant a\}, \\
(a, \infty) &= \{x \in \RR \colon x > a\}, & (-\infty, a] &= \{x \in \RR \colon x \leqslant a\}
\end{align*}
or \((-\infty, a) = \{x \in \RR \colon x < a\}\) with arbitrary \(a \in \RR\).
\end{remark}

\begin{remark}\label{r5.42}
In all topological spaces every single-point set and the empty set are connected (\cite[p.~108]{Bourbaki1995}). Moreover, if \((X, \tau)\) is a connected topological space, then the set \(X\) is the unique component of \((X, \tau)\). In particular, the empty set \(\varnothing\) is a component of \((X, \tau)\) if and only if \(X = \varnothing\).
\end{remark}

\begin{lemma}\label{l5.36}
Let \(A\) be a nonempty subset of \(\RR\). Then all components of \(A\) are intervals in \(\RR\).
\end{lemma}

\begin{proof}
Suppose contrary that there is a point \(p \in A\) for which the component \(C(p)\) of \(p\) in the topological space \(A\) is not an interval in \(\RR\). Then, by Lemma~\ref{l5.34}, the set \(C(p)\) is not a connected subset of \(\RR\). A subset \(E\) of \(\RR\) is connected if and only if \([x, y] \subseteq E\) holds whenever \(x\), \(y \in E\) and \(x < y\) (see, for example, Theorem~2.47 \cite{Rudin1976}). Consequently, there are \(x\), \(y \in C(p)\) and \(z \in \RR \setminus C(p)\) such that \(x < z < y\). Then the sets \((-\infty, z) \cap C(p)\) and \(C(p) \cap (z, \infty)\) are disjoint nonempty open subsets of the topological space \(C(p)\) and we have the equality
\[
C(p) = ((-\infty, z) \cap C(p)) \cup (C(p) \cap (z, \infty)).
\]
Hence, \(C(p)\) is not a connected subset of \(A\), that contradicts the definition of connected components.
\end{proof}

In what follows we will use the symbol \(\RR^{+} \setminus D(X)\) to denote the relative complement of the distance set \(D(X)\) of an ultrametric space \(X\) with respect to \(\RR^{+}\).

\begin{theorem}\label{t5.46}
Let \((X, d)\) be a nonempty ultrametric space and let \(D(X)\) be the distance set of \((X, d)\). Then the following conditions are equivalent.
\begin{enumerate}
\item \label{t5.46:s1} The space \((X, d)\) is bounded and \(\diam (X, d) \notin D(X)\).
\item \label{t5.46:s2} There is \(a \in (0, \infty)\) such that the set \([a, \infty)\) is a component of \(\RR^{+} \setminus D(X)\).
\item \label{t5.46:s3} There is \((Z, \delta) \in WS(X, d)\) such that if \(\Phi \colon Z \to X\) is a weak similarity and \(\psi \colon D(X) \to D(Z)\) is the scaling function of \(\Phi\), then 
\begin{equation}\label{t5.46:e1}
g|_{D(X)} \neq \psi
\end{equation}
holds for every pseudoultrametric preserving function \(g \colon \RR^{+} \to \RR^{+}\).
\end{enumerate}
\end{theorem}

\begin{proof}
The validity \(\ref{t5.46:s1} \Leftrightarrow \ref{t5.46:s2}\) follows directly from the definitions. 

Let us prove the validity of~\(\ref{t5.46:s1} \Leftrightarrow \ref{t5.46:s3}\). Let \(\ref{t5.46:s1}\) hold. Then, by Theorem~\ref{t2.35}, there is an unbounded \((Z, \delta) \in WS(X, d)\). Let \(\Phi \colon Z \to X\) be a weak similarity and let \(\psi \colon D(X) \to D(Z)\) be the scaling function of \(\Phi\). Since \((Z, \delta)\) is unbounded and \(\psi\) is surjective, the set \(\psi(D(X))\) is unbounded. Now if \(g \colon \RR^{+} \to \RR^{+}\) is an arbitrary pseudoultrametric preserving function, then we have the inequality
\[
g(t) \leqslant g(\diam X)
\]
for every \(t \in D(X)\), because \(g\) is increasing. Thus, 
\[
g(D(X)) \subseteq [0, g(\diam X)]
\]
holds. The last statement and unboundedness of \(\psi(D(X))\) imply \eqref{t5.46:e1}.

Let condition \(\ref{t5.46:s3}\) hold. If \(\ref{t5.46:s1}\) does not hold, then either \((X, d)\) is unbounded or \((X, d)\) is bounded and \(\diam (X, d) \in D(X)\).

Consider first the case when \((X, d)\) is unbounded. By condition~\(\ref{t5.46:s3}\), there is an ultrametric space \((Z, \delta) \in WS(X, d)\) such that if \(\Phi \colon Z \to X\) is a weak similarity with a scaling function \(\psi\), then \eqref{t5.46:e1} holds for every pseudoultrametric preserving \(g\). Let us define a function \(f \colon \RR^{+} \to \RR^{+}\) as
\begin{equation}\label{t5.46:e2}
f(t) = \sup \{\psi(s) \colon s \in [0, t] \cap D(X)\}, \quad t \in \RR^{+}.
\end{equation}
Then \(f\) is increasing and satisfies the equality \(f(0) = 0\). By Proposition~\ref{p5.45}, the function \(f\) is pseudoultrametric preserving. Moreover, from~\eqref{t5.46:e2} it follows the equality
\begin{equation}\label{t5.46:e3}
f|_{D(X)} = \psi.
\end{equation}
The last equality contradicts~\eqref{t5.46:e1} for \(g = f\).

Now let \((X, d)\) be bounded. If we have \(|X| = 1\), then condition~\(\ref{t5.46:s3}\) is false, contrary to supposition. If \(|X| \geqslant 2\) holds, then, as above, \(f\) defined by \eqref{t5.46:e2} is pseudoultrametric preserving and satisfies~\eqref{t5.46:e3}, which completes the proof of the theorem. 
\end{proof}

\begin{theorem}\label{t5.51}
Let \((X, d)\) be a nonempty ultrametric space and let \(D(X) = D(X, d)\) be the distance set of \((X, d)\). Then the following conditions are equivalent:
\begin{enumerate}
\item \label{t5.51:s1} There is \(a \in (0, \infty)\) such that \((0, a]\) or \([a, \infty)\) is a component of \(\RR^{+} \setminus D(X)\).
\item \label{t5.51:s2} There is an ultrametric space \((Z, \delta)\) such that \((Z, \delta)\) and \((X, d)\) are weakly similar, but if \(\Phi \colon Z \to X\) is a weak similarity and \(\psi \colon D(X) \to D(Z)\) is the scaling function of \(\Phi\), then 
\begin{equation}\label{t5.51:e1}
g|_{D(X)} \neq \psi
\end{equation}
holds for every ultrametric preserving function \(g \colon \RR^{+} \to \RR^{+}\).
\end{enumerate}
\end{theorem}

\begin{proof}
The theorem is evidently valid for the case \(|X| = 1\). So in what follows we will suppose that \(|X| \geqslant 2\).

\(\ref{t5.51:s1} \Rightarrow \ref{t5.51:s2}\). Let condition~\ref{t5.51:s1} hold. If there is \(a > 0\) such that \([a, \infty)\) is a component of \(\RR^{+} \setminus D(X)\), then \ref{t5.51:s2} takes place by Theorem~\ref{t5.46}.

If \((0, a]\) is a component of \(\RR^{+} \setminus D(X)\) for some \(a > 0\), then we define a function \(\psi\) on \(D(X) = D(X, d)\) as
\begin{equation}\label{t5.51:e2}
\psi(t) = \begin{cases}
0 & \text{if } t = 0,\\
t-a & \text{if } t \in D(X) \cap [a, \infty).
\end{cases}
\end{equation}
By Lemma~\ref{l5.22}, the mapping \(\rho \colon X \times X \to \RR^{+}\) with
\begin{equation}\label{t5.51:e3}
\rho(x, y) = \psi (d(x, y))
\end{equation}
is an ultrametric on \(X\). From \eqref{t5.51:e2} and \eqref{t5.51:e3} it follows that the diagram
\[
\ctdiagram{
\ctv 0,50:{X \times X}
\ctv 100,50:{X \times X}
\ctv 0,0:{D(X, \rho)}
\ctv 100,0:{D(X, d)}
\ctet 0,50,100,50:{\operatorname{Id} \otimes \operatorname{Id}}
\ctel 0,50,0,0:{\rho}
\cter 100,50,100,0:{d}
\ctet 100,0,0,0:{\psi}
}
\]
is commutative and \(\psi \colon D(X, d) \to D(X, \rho)\) is strictly increasing and bijective. Hence, \(\operatorname{Id} \colon X \to X\) is a weak similarity of the ultrametric spaces \((X, \rho)\), \((X, d)\) and \(\psi\) is the scaling function of this weak similarity. Suppose that there is an ultrametric preserving function \(g \colon \RR^{+} \to \RR^{+}\) such that 
\begin{equation}\label{t5.51:e4}
g|_{D(X, d)} = \psi.
\end{equation}
Since \((0, a]\) is a component of \(\RR^{+} \setminus D(X)\), there is a strictly decreasing sequence \((t_n)_{n \in \mathbb{N}} \subseteq D(X, d)\) for which
\begin{equation}\label{t5.51:e5}
\lim_{n \to \infty} t_n = a
\end{equation}
holds. Now using \eqref{t5.51:e2}, \eqref{t5.51:e4} and \eqref{t5.51:e5} we obtain the contradiction
\[
0 < g(a) \leqslant \lim_{n \to \infty} g(t_n) = \lim_{n \to \infty} (t_n - a) = 0.
\]
So, equality~\eqref{t5.51:e4} is not possible for any ultrametric preserving function \(g\).

\(\ref{t5.51:s2} \Rightarrow \ref{t5.51:s1}\). Let \ref{t5.51:s2} be valid. Suppose condition \ref{t5.51:s1} does not hold. Then the point \(0\) is an accumulation point of \(D(X)\) and either \(\diam X \in D(X)\) or \((X, d)\) is unbounded. By condition~\(\ref{t5.51:s2}\), there is an ultrametric space \((Z, \delta)\) such that \((X, d)\) and \((Z, \delta)\) are weakly similar, and if \(\Phi \colon Z \to X\) is a weak similarity with a scaling function \(\psi\), then \eqref{t5.51:e1} holds for every ultrametric preserving \(g\). 

Let us define \(g \colon \RR^{+} \to \RR^{+}\) by equality~\eqref{t5.46:e2} with \(f = g\). Then, as in the proof of Theorem~\ref{t5.46}, we obtain that \(g\) is pseudoultrametric preserving and the equality 
\begin{equation}\label{t5.51:e6}
g|_{D(X, d)} = \psi
\end{equation}
holds. To complete the proof it suffices to note that \(g\) is ultrametric preserving. Indeed, the inequality \(g(t) > 0\) holds for every \(t \in (0, \infty)\) because \(g\) is increasing, \(\psi(t) > 0\) takes place for every \(t \in D(X) \setminus \{0\}\) and the point \(0\) is an accumulation point of \(D(X)\). Hence, \(g\) is ultrametric preserving by Remark~\ref{r5.47}.
\end{proof}

Recall that an ultrametric space \((X, d)\) is \emph{spherically complete} if every sequence \((B_n)_{n\in \mathbb{N}}\) of closed balls 
\[
B_n = \{x \in X \colon d(x_n, x) \leqslant r_n\}, \quad r_n > 0, \quad x_n \in X,
\]
with \(B_1 \supseteq B_2 \supseteq \ldots\) has a nonempty intersection (see, for example, Definition~20.1 \cite{Sch1985}).

\begin{example}\label{ex5.52}
Let \(X\) be a nonempty subset of an interval \((a, b)\), \(0 < a < b < \infty\). Following~\cite{DLPS2008TaiA}, we define an ultrametric \(d \colon X \times X \to \RR^{+}\) as
\[
d(x, y) = \begin{cases}
0 & \text{if } x = y,\\
\max\{x, y\} & \text{if } x \neq y.
\end{cases}
\]
Then \((X, d)\) is spherically complete and has a nonempty diametrical graph if and only if, for every \((Z, \delta) \in WS(X, d)\) and every weak similarity \(\Phi \colon Z \to X\), there is an ultrametric preserving function \(g \colon \RR^{+} \to \RR^{+}\) such that \(g|_{D(X)} = \psi\), where \(\psi\) is the scaling function of \(\Phi\).
\end{example}

\begin{theorem}\label{t5.34}
Let \((X, d)\) be a nonempty ultrametric space and let \(D(X) = D(X, d)\) be the distance set of \((X, d)\). Then the following conditions are equivalent:
\begin{enumerate}
\item \label{t5.34:s1} Every component of \(\RR^{+} \setminus D(X, d)\) is either open in \(\RR\) or it is a single-point set.
\item \label{t5.34:s2} For every \((Z, \delta) \in WS(X, d)\) and every weak similarity \(\Phi \colon Z \to X\) there is a strictly increasing ultrametric preserving function \(g \colon \RR^{+} \to \RR^{+}\) such that
\begin{equation}\label{t5.34:e0}
g|_{D(X, d)} = \psi,
\end{equation}
where \(\psi\) is the scaling function of the weak similarity \(\Phi\).
\end{enumerate}
\end{theorem}

\begin{proof}
This theorem is obviously true for \(|X| = 1\). Furthermore, if the equality \(D(X) = \RR^{+}\) holds, then we have \(\RR^{+} \setminus D(X) = \varnothing\). Hence, in this case condition~\ref{t5.34:s1} is valid by Remark~\ref{r5.42}. Moreover, for the case \(D(X) = \RR^{+}\), the domain of the scaling function \(\psi\) of any weak similarity \(\Phi \colon Z \to X\) coincides with \(\RR^{+}\). By Definition~\ref{d2.34}, the function \(\psi\) is strictly increasing and satisfies \(\psi(0) = 0\). Hence, \(\psi\) is strictly increasing and ultrametric preserving by Theorem~\ref{t5.21}. So, in what follows, without loss of generality, we can assume that \(|X| \geqslant 2\) and \(\RR^{+} \setminus D(X) \neq \varnothing\).

\(\ref{t5.34:s1} \Rightarrow \ref{t5.34:s2}\). Suppose that condition~\ref{t5.34:s1} holds. Let \((Z, \delta) \in WS(X, d)\) and let \(\Phi \colon Z \to X\) be a weak similarity with a scaling function \(\psi \colon D(X) \to D(Z)\),
\[
\ctdiagram{
\ctv 0,50:{Z \times Z}
\ctv 100,50:{X \times X}
\ctv 0,0:{D(Z, \delta)}
\ctv 100,0:{D(X, d)}
\ctet 0,50,100,50:{\Phi \otimes \Phi}
\ctel 0,50,0,0:{\delta}
\cter 100,50,100,0:{d}
\ctet 100,0,0,0:{\psi}
}.
\]
We want to continue the scaling function \(\psi \colon D(X) \to D(Z)\) to a strictly increasing ultrametric preserving function \(g \colon \RR^{+} \to \RR^{+}\). 

Let \(t \in \RR^{+}\). 

Consider first the case when \(t \in D(X)\) or \(\{t\}\) is a component of \(\RR^{+} \setminus D(X)\). Then we have
\[
D(X) \cap [0, t] \neq \varnothing \neq D(X) \cap [t, \infty).
\]
Consequently, there is a point \(t^{*} \in \RR^{+}\) such that
\begin{equation}\label{t5.34:e2}
\sup\{\psi(s) \colon s \in D(X) \cap [0, t]\} \leqslant t^{*} \leqslant \inf\{\psi(s) \colon s \in D(X) \cap [t, \infty)\}.
\end{equation}
Let us define \(g(t)\) as
\begin{equation}\label{t5.34:e5}
g(t) = t^{*}.
\end{equation}

If there is an interval \((a, b) \ni t\) such that \((a, b)\) is a component of \(\RR^{+} \setminus D(X)\), then the points \(a\) and \(b\) belong to \(D(X)\) and we can define \(g(t)\) as
\begin{equation}\label{t5.34:e3}
g(t) = \psi(a) + \frac{\psi(b) - \psi(a)}{b - a} (t - a).
\end{equation}

Similarly, if \(t \in (a, \infty)\), where \((a, \infty)\) is a component of \(\RR^{+} \setminus D(X)\), then \(a \in D(X)\). From \(|X| \geqslant 2\) and Lemma~\ref{l5.29} it follows that the inequalities \(a > 0\) and \(\psi(a)  > 0\) hold. In this case we write
\begin{equation}\label{t5.34:e4}
g(t) = \frac{\psi(a)}{a} t.
\end{equation}

Since every point of \(\RR^{+} \setminus D(X)\) belongs to the component of this point, condition~\ref{t5.34:s1} implies that \(g(t)\) is defined now for all \(t \in \RR^{+} \setminus D(X)\). It should be noted here that the components of two distinct points of any topological space either coincide or are disjoint, so that this definition is correct. 

If \(t \in D(X)\), then we have
\[
\psi(t) = \sup\{\psi(s) \colon s \in D(X) \cap [0, t]\}
\]
and
\[
\psi(t) = \inf\{\psi(s) \colon s \in D(X) \cap [t, \infty)\}
\]
because \(\psi\) is increasing. Consequently, the equality \(\psi(t) = g(t)\) follows from~\eqref{t5.34:e2} and \eqref{t5.34:e5} for every \(t \in D(X)\), that implies~\eqref{t5.34:e0} and the equality \(g(0) = 0\). 

Let us prove the validity of the implication
\begin{equation}\label{t5.34:e9}
(t_1 < t_2) \Rightarrow (g(t_1) < g(t_2))
\end{equation}
for all \(t_1\), \(t_2 \in \RR^{+}\). 

First of all, it should be noted that \eqref{t5.34:e0} implies \eqref{t5.34:e9} for all \(t_1\), \(t_2 \in D(X)\) since the scaling function \(\psi\) is strictly increasing. Let \(t_1 \in \RR^{+} \setminus D(X)\), \(t_2 \in D(X)\) and \(t_1 < t_2\). If \(\{t_1\}\) is a component of \(\RR^{+} \setminus D(X)\), then there is a point \(t_3 \in D(X) \cap [t_1, \infty)\) such that \(\psi(t_3) < \psi(t_2)\). From \eqref{t5.34:e2} and \eqref{t5.34:e5} it follows 
\[
g(t_1) \leqslant g(t_3) = \psi(t_3),
\]
which together with \(\psi(t_3) < \psi(t_2) = g(t_2)\) implies \(g(t_1) < g(t_2)\). The case when \(t_1 \in D(X)\), \(t_2 \in \RR^{+} \setminus D(X)\) and \(t_1 < t_2\) can be considered similarly.

Let \(t_1\) and \(t_2\) be points of \(\RR^{+} \setminus D(X)\) such that \(t_1 < t_2\). If the components of these points are the same, then \eqref{t5.34:e9} follows from \eqref{t5.34:e3} and \eqref{t5.34:e4}. Suppose now that the points \(t_1\) and \(t_2\) belong to the different components of \(\RR^{+} \setminus D(X)\). Then there is a point \(t_3 \in D(X)\) such that
\[
t_1 < t_3 < t_2.
\]
It was shown above that in this case we have
\[
g(t_1) < g(t_3) < g(t_2).
\]
Thus, \eqref{t5.34:e9} is valid for all \(t_1\), \(t_2 \in \RR^{+}\).

For every strictly increasing function \(f \colon \RR^{+} \to \RR^{+}\) satisfying \(f(0) = 0\), we have \(f^{-1}(0) = \{0\}\). Thus, the function \(g\) defined above is strictly increasing and ultrametric preserving by Pongsriiam---Termwuttipong theorem.

\(\ref{t5.34:s2} \Rightarrow \ref{t5.34:s1}\). Suppose condition \ref{t5.34:s1} is false. We must show that \ref{t5.34:s2} is also false. 

The classification of intervals of \(\RR\) given in Remark~\ref{r5.35} and Lemma~\ref{l5.36} imply that at least one from the following statements are valid:
\begin{enumerate}
\item [\((s_1)\)] There is \(a \in (0, \infty)\) such that \((0, a]\) or \([a, \infty)\) is a component of \(\RR^{+} \setminus D(X)\).
\item [\((s_2)\)] There are some points \(a\), \(b\) such that \(0 < a < b < \infty\) and \([a, b)\) or \((a, b]\) is a component of \(\RR^{+} \setminus D(X)\).
\end{enumerate}

If \((s_1)\) holds, then \ref{t5.34:s2} is false by Theorem~\ref{t5.51}.

Let us consider the case, when \((s_2)\) holds and \([a, b)\) is a component of the set \(\RR^{+} \setminus D(X)\). Let us consider the function \(\psi\) defined as
\begin{equation}\label{t5.34:e6}
\psi(t) = \begin{cases}
t &\text{if } t \in D(X) \cap [0, a),\\
t - (b-a) &\text{if } t \in D(X) \cap [b, \infty).
\end{cases}
\end{equation}
From \([a, b) \cap D(X) = \varnothing\) it follows that \(\psi\) is defined for all \(t \in D(X)\). It is clear that \(\psi\) is strictly increasing on \(D(X)\) and \(\psi(t) = 0\) holds if and only if \(t = 0\). By Lemma~\ref{l5.22}, the mapping 
\[
X \times X \ni \<x, y> \mapsto \psi(d(x, y))
\]
is an ultrametric on \(X\). 

Write \(\delta = \psi \circ d\). Then the ultrametric spaces \((X, d)\) and \((X, \delta)\) are weakly similar, the identical mapping \(\operatorname{Id} \colon X \to X\) is a weak similarity of \((X, \delta)\) and \((X, d)\), and the function \(\psi\),
\[
D(X, d) \ni t \mapsto \psi(t) \in D(X, \delta),
\]
is the scaling function of this weak similarity.

We claim that there is no a strictly increasing ultrametric preserving function \(g \colon \RR^{+} \to \RR^{+}\) for which the equality
\begin{equation}\label{t5.34:e10}
g|_{D(X, d)} = \psi
\end{equation}
holds. To prove it, we may consider a point \(b_1 \in [a, b)\) such that
\begin{equation}\label{t5.34:e7}
a < b_1 < b.
\end{equation}
If \(g \colon \RR^{+} \to \RR^{+}\) is an arbitrary strictly increasing ultrametric preserving function, then \eqref{t5.34:e7} implies
\begin{equation}\label{t5.34:e8}
g(a) < g(b_1) < g(b).
\end{equation}
Since \([a, b)\) is a component of \(\RR^{+} \setminus D(X, d)\), the membership \(b \in D(X, d)\) and the equality \(a > 0\) hold. Hence, there is a strictly increasing sequence \((a_n)_{n \in \mathbb{N}}\) of points of the set \(D(X, d) \cap [0, a)\) such that
\[
\lim_{n \to \infty} a_n = a.
\]
The last equality, \eqref{t5.34:e10} and \eqref{t5.34:e6} imply that
\[
\lim_{n \to \infty} g(a_n) = \lim_{n \to \infty} \psi(a_n) = a.
\]
Moreover, \eqref{t5.34:e10}, \eqref{t5.34:e6} and \(b \in D(X, d)\) imply that \(\psi(b) = g(b) = b - (b-a) = a\). Now using \eqref{t5.34:e8} we have the contradiction
\[
a = \lim_{n \to \infty} g(a_n) \leqslant g(a) < g(b_1) < g(b) = a.
\]

If \((s_2)\) holds for a component \((a, b]\) of \(\RR^{+} \setminus D(X, d)\), then using the function
\[
\varphi(t) = \begin{cases}
t & \text{if } t \in D(X) \cap [0, a],\\
t - (b-a) & \text{if } t \in D(X) \cap (b, \infty)
\end{cases}
\]
instead of the function \(\psi\) defined by \eqref{t5.34:e6} and arguing as above we can see that the equality
\[
g|_{D(X, d)} = \varphi
\]
is impossible for any strictly increasing ultrametric preserving function \(g\).
\end{proof}

\begin{corollary}\label{c5.38}
Let \(X\) and \(Z\) be nonempty weakly similar ultrametric spaces and let the distance set \(D(X)\) be a closed subset of \(\RR^{+}\). Then, for every weak similarity \(\Phi \colon Z \to X\) of \(Z\) and \(X\), the scaling function of \(\Phi\) admits a continuation to a strictly increasing ultrametric preserving function.
\end{corollary}

\begin{proof}
It follows from Theorem~\ref{t5.34} because the set \(\RR^{+} \setminus D(X)\) is an open subset of \(\RR\) and every component of any nonempty open subset of \(\RR\) is an open interval.
\end{proof}

Recall that a metric space \((X, d)\) is \emph{totally bounded} if for every \(\varepsilon > 0\) there is a finite set \(\{x_1, \ldots, x_n\} \subseteq X\) such that 
\[
\min_{1 \leqslant i \leqslant n} d(x_i, x) < \varepsilon
\]
for every \(x \in X\).

The following lemma gives us a characterization of distance sets of totally bounded ultrametric spaces.

\begin{lemma}\label{c6.9}
The following statements are equivalent for every \(A \subseteq \RR^{+}\):
\begin{enumerate}
\item \label{c6.9:s1} There is an infinite compact ultrametric space \(X\) such that \(A = D(X)\).
\item \label{c6.9:s2} There is an infinite totally bounded ultrametric space \(X\) such that \(A = D(X)\).
\item \label{c6.9:s3} There is a strictly decreasing sequence \((x_n)_{n \in \mathbb{N}}\) of positive real numbers such that 
\[
\lim_{n \to \infty} x_n = 0
\]
holds and the equivalence 
\[
(x \in A) \Leftrightarrow (x = 0 \text{ or } \exists n \in \mathbb{N} \colon x_n = x)
\]
is valid for every \(x \in \RR^{+}\).
\end{enumerate}
\end{lemma}

\begin{proof}
The validity of \(\ref{c6.9:s1} \Rightarrow \ref{c6.9:s3}\) follows from Proposition~19.2 \cite[p.~51]{Sch1985}. 

Since a metric space is totally bounded if and only if the completion of this space is compact (see Corollary~4.3.30 in \cite{Eng1989}), the principle: ``No new values of the ultrametric after completion'' implies the validity of \(\ref{c6.9:s1} \Leftrightarrow \ref{c6.9:s2}\). This principle can be found, for example, in \cite[p.~4]{PerezGarcia2010}.

To complete the proof it suffices to note that if \(A\) satisfies condition \ref{c6.9:s3}, \(X = A\) and \(d \colon X \times X \to \RR^{+}\) is an ultrametric defined as
\[
d(x, y) = \begin{cases}
0 & \text{if } x = y,\\
\max\{x, y\} & \text{if } x \neq y,
\end{cases}
\]
then \((X, d)\) is a compact ultrametric space for which \(D(X, d) = A\) holds.
\end{proof}

\begin{corollary}\label{c5.35}
Let \(X\) be a nonempty totally bounded ultrametric space. Then, for every weak similarity \(\Phi \colon Z \to X\) of an ultrametric space \(Z\) and the ultrametric space \(X\), there is a strictly increasing ultrametric preserving function \(g \colon \RR^{+} \to \RR^{+}\) such that
\[
g|_{D(X)} = \psi,
\]
where \(\psi \colon D(X) \to D(Z)\) is the scaling function of \(\Phi\).
\end{corollary}

\begin{proof}
It follows from Corollary~\ref{c5.38} and Lemma~\ref{c6.9}.
\end{proof}

\begin{example}\label{ex5.30}
Let \(X\) be a subset of \([0, 1]\) such that
\[
X = \left\{\frac{1}{n} \colon n \in \mathbb{N}\right\}
\]
and let \(d \colon X \times X \to \RR^{+}\) and \(\delta \colon X \times X \to \RR^{+}\) be ultrametrics on \(X\) defined for all \(x\), \(y \in X\) as
\[
d(x, y) = \begin{cases}
0 & \text{if } x = y,\\
\max\{x^2, y^2\} & \text{if } x \neq y
\end{cases}
\]
and, respectively,
\begin{equation}\label{ex5.30:e1}
\delta(x, y) = \begin{cases}
0 & \text{if } x = y,\\
1 + \max\{x, y\} & \text{if } x \neq y.
\end{cases}
\end{equation}
Then \((X, d)\) and \((X, \delta)\) are weakly similar, the identical mapping \(\operatorname{Id} \colon X \to X\) is a (unique) weak similarity of \((X, d)\) and \((X, \delta)\), and the equality \(d(x, y) = f(\delta(x, y))\) holds for all \(x\), \(y \in X\) with the scaling function \(f \colon D(X, \delta) \to D(X, d)\) such that
\begin{equation}\label{ex5.30:e2}
D(X, d) = \{0\} \cup \left\{\frac{1}{n^2} \colon n \in \mathbb{N}\right\}, \quad D(X, \delta) = \{0\} \cup \left\{1 + \frac{1}{n} \colon n \in \mathbb{N}\right\}
\end{equation}
and
\begin{equation}\label{ex5.30:e3}
f(t) = \begin{cases}
0 & \text{if } t = 0,\\
(t-1)^2 & \text{if } t \in D(X, \delta) \setminus \{0\}.
\end{cases}
\end{equation}
It is interesting to note that the ultrametric space \((X, \delta)\) is not totally bounded because \(|X| = \aleph_0\) and \(\delta(x, y) \geqslant 1\) whenever \(x \neq y\). In addition, there is no ultrametric preserving function \(\varphi \colon \RR^{+} \to \RR^{+}\) such that
\begin{equation}\label{ex5.30:e4}
\varphi|_{D(X, \delta)} = f.
\end{equation}
Indeed, if \(\varphi \colon \RR^{+} \to \RR^{+}\) is increasing and \eqref{ex5.30:e4} holds, then using \eqref{ex5.30:e2} and \eqref{ex5.30:e3} we obtain the equality \(\varphi(t) = 0\) for all \(t \in [0, 1]\). Thus, \(\varphi\) is not ultrametric preserving by Pongsriiam---Termwuttipong theorem. 

By Lemma~\ref{l5.29}, there is the inverse function \(\psi \colon D(X, d) \to D(X, \delta)\) for the scaling function \(f\). It is easy to show that
\begin{equation}\label{ex5.30:e5}
\psi(s) = \begin{cases}
0 & \text{if } s = 0,\\
1 + \sqrt{s} & \text{if } s \in D(X, d) \setminus \{0\}.
\end{cases}
\end{equation}
In the proof of Theorem~\ref{t5.34} it was used the piecewise linear expansion of \(\psi\) to a strictly increasing ultrametric preserving function \(g\) depicted in Figure~\ref{fig14} below.
\end{example}

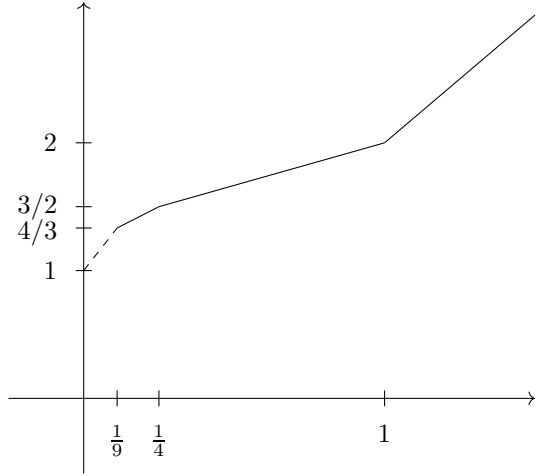
\begin{figure}[ht]
\begin{center}
\begin{tikzpicture}
\def\dx{4cm}
\def\dy{1.7cm}
\def\de{3pt}
\draw [->] (-1, 0) -- (1.5*\dx, 0);
\draw [->] (0, -1) -- (0, 3.1*\dy);
\draw (\dx, -\de) node [label = below:{\(1\)}] {} -- (\dx, \de);
\draw (-\de, 2*\dy) node [label = left:{\(2\)}] {} -- (\de, 2*\dy);
\draw (\dx, 2*\dy) -- (1.5*\dx, 3*\dy);

\draw (\dx/4, -\de) node [label = below:{\(\frac{1}{4}\)}] {} -- (\dx/4, \de);
\draw (\dx/4, 3*\dy/2) -- (\dx, 2*\dy);
\draw (-\de, 3*\dy/2) node [label = left:{\(3/2\)}] {} -- (\de, 3*\dy/2);

\draw (\dx/9, -\de) node [label = below:{\(\frac{1}{9}\)}] {} -- (\dx/9, \de);
\draw (\dx/9, 4*\dy/3) -- (\dx/4, 3*\dy/2);
\draw (-\de, 4*\dy/3) -- (\de, 4*\dy/3);
\draw (-\de, 3.9*\dy/3) node [label = left:{\(4/3\)}] {};

\draw [dashed] (0, \dy) -- (\dx/9, 4*\dy/3);
\draw (-\de, \dy) node [label = left:{\(1\)}] {} -- (\de, \dy);
\end{tikzpicture}
\end{center}
\caption{The graph of strictly increasing ultrametric preserving function \(g \colon \RR^{+} \to \RR^{+}\) corresponding to the scaling function \(\psi\) defined by rule \eqref{ex5.30:e5}.}
\label{fig14}
\end{figure}

Theorems~\ref{t5.46}--\ref{t5.34} can be considered as reformulations of some results guaranteeing, in the language of ultrametric spaces, the existence of isotonic extension of given isotonic mappings. For example, Theorem~\ref{t5.46} is a simple corollary of the following.

\begin{theorem}\label{t5.57}
Let \(D\) be a subset of \(\RR^{+}\) such that \(0 \in D\). Then the following conditions are equivalent:
\begin{enumerate}
\item \label{t5.57:s1} For every strictly increasing function \(\psi \colon D \to \RR^{+}\) with \(\psi(0) = 0\) there is an increasing function \(g \colon \RR^{+} \to \RR^{+}\) such that \(g|_{D} = \psi\).
\item \label{t5.57:s2} The poset \((D, {\leqslant})\) has a largest element or \(D\) is a confinal subset of \((\RR^{+}, {\leqslant})\). 
\end{enumerate}
\end{theorem}

The above theorem admits also future generalizations for partially ordered sets and, in particular, for different types of semilattices and lattices \cite{GLPFM1970, SorMSN1952, DovJMSUS2020, BGAU2012, BNSPRSESA2003, FofMN1969}. We note only that lattices are a special kind of ordered sets which play a fundamental role in the theory of ultrametric spaces \cite{Lem2003AU}.

In conclusion, we want to formulate an open problem, which is in a sense dual to Problem~\ref{pr1.13}.

\begin{problem}\label{pr2.20}
Describe the ultrametric spaces \((X, d)\) for which every \((Z, \delta) \in WS(X, d)\) and every weak similarity \(\Phi \colon X \to Z\) admit an ultrametric preserving (pseudoultrametric preserving or strictly increasing ultrametric preserving) function \(f \colon \RR^{+} \to \RR^{+}\) such that the diagram
\begin{equation*}
\ctdiagram{
\ctv 0,25:{Z \times Z}
\ctv 100,25:{D(Z)}
\ctv 200,25:{\RR^{+}}
\ctv 0,-25:{X \times X}
\ctv 100,-25:{D(X)}
\ctv 200,-25:{\RR^{+}}
\ctet 0,25,100,25:{\delta}
\ctet 100,25,200,25:{\Id_{D(Z)}}
\ctet 0,-25,100,-25:{d}
\ctet 100,-25,200,-25:{\Id_{D(X)}}
\ctel 0,-25,0,25:{\Phi \otimes \Phi}
\ctel 100,25,100,-25:{\psi}
\cter 200,25,200,-25:{f}
}
\end{equation*}
is commutative, where \(\psi\) is the scaling function of \(\Phi\) and \(\Id_{D(X)}\), \(\Id_{D(Z)}\) are the identical embeddings of the distance sets \(D(X)\) and, respectively, \(D(Z)\) in \(\RR^{+}\).
\end{problem}


\end{document}